\documentclass[12pt,reqno]{amsart}

\usepackage{amssymb,amsmath,graphicx,amsfonts,euscript,mathrsfs}
\usepackage{color}
\usepackage{hyperref}

\numberwithin{equation}{section}

\def\H{{\cal H}}

\def\R{\mathbb{R}}
\def\Z{\mathbb{Z}}
\def\T{\mathbb{T}}
\def\C{\mathbb{C}}

\def\H1{H^1(\R)}

\renewcommand{\theequation}{\arabic{section}.\arabic{equation}}
\newtheorem{thm}{Theorem}
\newtheorem{lem}{Lemma}

\newtheorem{prop}{Proposition}

\newtheorem{defn}{Definition}

\newtheorem{remark}{Remark}

\makeatletter
\newcommand{\Extend}[5]{\ext@arrow0099{\arrowfill@#1#2#3}{#4}{#5}}
\makeatother

\begin{document}

\setcounter{page}{1}

\title[Solitary wave solutions for gDNLS]{Instability of the solitary wave solutions for the generalized derivative nonlinear Schr\"odinger equation in the endpoint case}

\author{Bing Li}
\address{Center for Applied Mathematics\\
Tianjin University\\
Tianjin 300072, China}
\email{binglimath@gmail.com}
\thanks{}

\author{Cui Ning}
\address{School of Mathematics, South China University of Technology, Guangzhou, Guangdong 510640, P.R.China}
\email{cuiningmath@gmail.com}
\thanks{}



\keywords{generalized DNLS, orbital instability, solitary wave solutions, endpoint case}

\maketitle

\begin{abstract}\noindent
We consider the stability theory of solitary wave solutions for the generalized derivative nonlinear Schr\"odinger equation
  $$
  i\partial_{t}u+\partial_{x}^{2}u+i|u|^{2\sigma}\partial_x u=0,
  $$
  where $1<\sigma<2$.
  The equation has a two-parameter family of solitary wave solutions of the form
  $$
u_{\omega,c}(t,x)=e^{i\omega t+i\frac c2(x-ct)-\frac{i}{2\sigma+2}\int_{-\infty}^{x-ct}\varphi^{2\sigma}_{\omega,c}(y)dy}\varphi_{\omega,c}(x-ct).
  $$
  The stability theory in the frequency region of $|c|<2\sqrt{\omega}$ was studied previously. In this paper, we prove the instability of the solitary wave solutions in the endpoint case $c=2\sqrt{\omega}$.
\end{abstract}

\section{Introduction}

\subsection{Setting of the Problem}

In this paper, we consider the stability theory of solitary wave solutions for the generalized derivative nonlinear Schr\"odinger (gDNLS) equation:
\begin{equation}\label{eqs:gDNLS}
   \left\{ \aligned
    &i\partial_{t}u+\partial_{x}^{2}u+i|u|^{2\sigma}\partial_x u=0,\qquad (t,x)\in [0,T)\times \R,\\
    &u(0,x)=u_0(x),\qquad \qquad \qquad\ \ \ x\in \R,
   \endaligned
  \right.
\end{equation}
where $\sigma>0$.

When $\sigma=1$, by a suitable gauge transformation, \eqref{eqs:gDNLS} is transformed to the standard derivative nonlinear Schr\"odinger (DNLS) equation:
\begin{equation}\label{DNLS0}
   \left\{ \aligned
    &i\partial_{t}u+\partial_{x}^{2}u+i\partial_x (|u|^{2}u)=0,\qquad (t,x)\in [0,T)\times \R,\\
    &u(0,x)=u_0(x),\qquad \qquad \qquad\ \ x\in \R.
   \endaligned
  \right.
\end{equation}
It describes an Alfv\'{e}n wave and appears in plasma physics, nonlinear optics, and so on (see \cite{MOMT-PHY, M-PHY}). The Cauchy problem of  \eqref{DNLS0} is local well-posedness in the energy space $\H1$ by Hayashi and Ozawa \cite{HaOz-92-DNLS, HaOz-94-DNLS}. That is,  given $u_0\in \H1$, there exists a unique maximal solution $u(t,x)$ of \eqref{DNLS0} in $C([0,T), \H1)$, moreover, $\lim_{t\rightarrow T}\|u\|_{L^2}=\infty$ if $T<+\infty$. See also \cite{TuFu-80-DNLS, TuFu-81-DNLS,GuTa91, Oz-96-DNLS, Ta-99-DNLS-LWP,BiLi-01-Illposed-DNLS-BO,Ta-01-DNLS-GWP} for some of the previous or extended results.
Meanwhile, the global well-posedness was widely studied. In \cite{HaOz-92-DNLS}, the authors showed that $\H1$ initial data with $\|u_0\|_{L^2}<\sqrt{2\pi}$ gives global and $\H1$ bounded solutions. Recently, the global well-posedness with $\H1$ initial data satisfying $\|u_0\|_{L^2}<2\sqrt{\pi}$, has been established in the works by Wu \cite{Wu1, Wu2}. In \cite{Guo-Wu-15-DNLS}, a two-page's proof was presented by simplifying the argument in \cite{Wu2}.  More recently, Jenkins, Liu, Perry and Sulem \cite{JeLiPeSu} proved that the Cauchy problem \eqref{DNLS0} is global well-posedness for any $H^{2,2}$-initial datum, without mass restriction.
See also \cite{HaOz-94-DNLS, Oz-96-DNLS, CKSTT-01-DNLS, CKSTT-02-DNLS,Miao-Wu-Xu:2011:DNLS,GHLN,Wu1,Guo-Wu-15-DNLS} for the related results.

It is known (see for examples \cite{GuWu95,CoOh-06-DNLS,Wu2}) that \eqref{DNLS0} has a two-parameter family of solitary wave solutions:
\begin{align*}
\widetilde{u}_{\omega,c}(t,x)=e^{i\omega t+i\frac{c}{2}(x-ct)
-\frac{3}{4}i\int_{-\infty}^{x-ct}|\widetilde{\varphi}_{\omega,c}(\eta)|^2d\eta}\widetilde{\varphi}_{\omega,c}(x-ct),
\end{align*}
where $\omega=(\omega,c)\in\{(\omega,c)\in\mathbb{R}^+\times\mathbb{R}: c^2< 4\omega\ \mbox{or}\ c=2\sqrt{\omega}\}$, and $\widetilde{\varphi}_{\omega,c}$ is the solution of
$$-\partial_x^2\widetilde{\varphi}+(\omega-\frac{c^2}{4})\widetilde{\varphi}+\frac{c}{2}|\widetilde{\varphi}|^2\widetilde{\varphi}-\frac{3}{16}| \widetilde{\varphi}|^4\widetilde{\varphi}=0.$$

For general $\sigma>0$, \eqref{eqs:gDNLS} is regarded as an extension of (DNLS) equation. In energy space $\H1$, Hayashi and Ozawa \cite{HaOz-16-DNLS} proved that the Cauchy problem \eqref{eqs:gDNLS} is local well-posedness for $\sigma>1$.  See also \cite{Hao, santos, HaOz-16-DNLS, LiPoSa-Local-DNLS} for the related results.
Moreover, the $H^1$-solution $u(t)$ of \eqref{eqs:gDNLS} satisfies three conservation laws:
$${E}(u(t))={E}(u_0),\qquad {P}(u(t))={P}(u_0),\qquad {M}(u(t))={M}(u_0),$$
for all $t\in [0,T)$, where
\begin{align*}
E(u)=&\frac{1}{2}\| \partial_xu\|_{L^2}^2-\frac{1}{2(\sigma+1)}\mbox{Im}\int_{\mathbb{R}}|u|^{2\sigma}u\,\overline{\partial_xu}\,dx,\\
P(u)=&\frac{1}{2}(i\partial_xu,u)_{L^2}=\frac{1}{2}\mbox{Im}\int_{\mathbb{R}}u\,\overline{\partial_xu}\,dx,\\
M(u)=&\frac{1}{2}\|u\|_{L^2}^2.
\end{align*}
By using the conservation laws, Fukaya, Hayashi and Inui \cite{FuHaIn-16-DNLS} investigated the global well-posedness of \eqref{eqs:gDNLS} in energy space $\H1$ for $\sigma>1$ with some suitable size restriction on the initial datum. Hayashi and Ozawa \cite{HaOz-16-DNLS} proved  the global existence (without uniqueness) in $H^1(\R)$ for $0<\sigma<1$.

Equation \eqref{eqs:gDNLS} also admits a two-parameter family of solitary wave solutions:
\begin{align*}
u_{\omega,c}(t,x)=e^{i\omega t}\phi_{\omega,c}(x-ct),
\end{align*}
where $\phi_{\omega,c}(x)=\varphi_{\omega,c}(x)e^{i\frac{c}{2}x-\frac{i}{2(\sigma+1)}\int_{-\infty}^{x}\varphi_{\omega,c}^{2\sigma}(y)dy}$. Here $\varphi_{\omega,c}$ is the solution of
\begin{align*}
-\partial_x^2\varphi+(\omega-\frac{c^2}{4})\varphi+\frac{c}{2}\varphi^{2\sigma+1}
-\frac{2\sigma+1}{(2\sigma+2)^2}\varphi^{4\sigma+1}=0.
\end{align*}

\subsection{Stability theory of gDNLS}
In the continuation of these works, there are many results about the stability theory of solitary wave solutions for the generalized derivative nonlinear Schr\"odinger equation.

When $\sigma=1$, Guo and Wu \cite{GuWu95} proved that $\widetilde{u}_{\omega,c}(t,x)$ is stable for $c<0$ and $c^2< 4\omega$. Further, Colin and Ohta \cite{CoOh-06-DNLS} proved that $\widetilde{u}_{\omega,c}(t,x)$ is stable for $c^2< 4\omega$. The endpoint case $c = 2\sqrt{\omega }$ was studied by Kwon and Wu \cite{Soonsik-W-2014}. Recently, the stability of the multi-solitons is studied by Le Coz, Wu \cite{Stefan-W-15-MultiSoliton-DNLS} (see also  Miao, Tang, Xu \cite{Miao-Tang-Xu:2016:DNLS} in the two-solitons case).

In an effort to understand the stability theory of (DNLS) equation, one may add a term $"b|u|^4u"$ with $b>0$ to (DNLS) equation, which brings some destabilizing effect. In this case, Ohta \cite{Ohta-14} showed that there exists $\kappa\in (0,1)$ such that $\widetilde{u}_{\omega,c}(t,x)$ is stable when $-2\sqrt{\omega}<c <2\kappa\sqrt{\omega }$, and unstable when $2\kappa\sqrt{\omega}<c <2\sqrt{\omega }$. Moreover, Ning, Ohta and Wu \cite{NOW16, NOW17} proved $\widetilde{u}_{\omega,c}(t,x)$ was unstable both in the borderline case $c = 2\kappa\sqrt{\omega }$ and in the endpoint case $c = 2\sqrt{\omega }$.

When $0<\sigma<1$, Liu, Simpson and Sulem \cite{LiSiSu1} proved that the solitary wave solution $u_{\omega,c}(t,x)$ is stable for any $-2\sqrt\omega<c<2\sqrt\omega$; when $\sigma\ge2$, the solitary wave solution $u_{\omega,c}(t,x)$ is unstable for any $-2\sqrt\omega<c<2\sqrt\omega$. Recently, Guo \cite{guo-DNLS} proved the stability of the solitary wave solutions in the endpoint case $c=2\sqrt\omega$, $\sigma\in(0,1)$.

When $1<\sigma<2$, Liu, Simpson and Sulem \cite{LiSiSu1} proved that there exists $z_0(\sigma)\in (0,1)$ such that the solitary wave solution $u_{\omega,c}(t,x)$ is stable when $-2\sqrt\omega<c<2z_0\sqrt\omega$, and unstable when $2z_0\sqrt\omega<c<2\sqrt\omega$.  Further, Fukaya \cite{Fu-16-DNLS} proved that the solitary waves solution $u_{\omega,c}(t,x)$ is unstable when $\frac76< \sigma<2$, $c=2z_0\sqrt\omega$. Moreover, Guo, Ning and Wu \cite{Guo-Ning-Wu-18-gDNLS} and Miao, Tang and Xu \cite{Miao-Tang-Xu18} independently proved that the solitary waves solution $u_{\omega,c}(t,x)$ is unstable for any $1<\sigma<2$ in borderline case $c=2z_0\sqrt\omega$. After these works, the stability theory when $c=2\sqrt\omega$, $\sigma\in(1,2)$ is unsolved.

\subsection{Statement of the results}
In this paper, we aim to the unsolved case
$$c=2\sqrt\omega, \qquad \sigma\in(1,2).$$
More precisely, let us define
\begin{align*}
u_c(t,x)=e^{i\frac{c^2}{4}t}\phi_{\frac{c^2}{4},c}(x-ct),
\end{align*}
where
\begin{align}\label{0.1}
\phi_{\frac{c^2}{4},c}(x)=\varphi_{\frac{c^2}{4},c}(x)e^{i\frac{c}{2}x-\frac{i}{2(\sigma+1)}\int_{-\infty}^{x}\varphi_{\frac{c^2}{4},c}^{2\sigma}(y)dy},
\end{align}
and
\begin{align}\label{0.2}
\varphi_{\frac{c^2}{4},c}(x)=\bigg(\frac{2c(\sigma+1)}{\sigma^2(cx)^2+1}\bigg)^\frac{1}{2\sigma}.
\end{align}
For simplicity, we denote $\varphi_{c}$ to be $\varphi_{\frac{c^2}{4},c}$ and $\phi_{c}$ to be $\phi_{\frac{c^2}{4},c}$ for short.
Note that $\phi_c$ is the solution of
\begin{align}\label{Elliptic-comp}
-\partial_x^2\phi+\frac{c^2}{4}\phi+ci\partial_x\phi-i|\phi|^{2\sigma}\partial_x\phi=0,
\end{align}
and
 $\varphi_c$ is the solution of
\begin{align}\label{Elliptic}
-\partial_x^2\varphi+\frac{c}{2}\varphi^{2\sigma+1}-\frac{2\sigma+1}{(2\sigma+2)^2}\varphi^{4\sigma+1}=0.
\end{align}

\begin{remark}\label{rem:1}
Compared with the case of $-2\sqrt\omega<c<2\sqrt\omega$,  the solution of the elliptic equation \eqref{Elliptic} $\varphi_c$ is "zero mass" in the endpoint case $c=2\sqrt\omega$. For the (DNLS) equation, there also appears "zero mass" in the endpoint case $c=2\sqrt\omega$, see \cite{NOW16, Wu2}.

From \eqref{0.2}, we know that $\phi_{c},\varphi_{c}\notin L^2(\R)$, when $\sigma\ge 2$. Hence, compared to the definitions of stability/instability  in the following, the analogous definitions should be given in a different way in the case of $\sigma\ge 2$.
\end{remark}

For $\varepsilon>0$, we define
$$ U_\varepsilon(\phi_c)
=\{u\in H^1(\mathbb{R}): \inf_{(\theta,y)\in\mathbb{R}^2}\|u-e^{i\theta}\phi_c(\cdot-y)\|_{H^1}<\varepsilon\}.$$

\begin{defn}
We say that the solitary wave solution $e^{i\frac{c^2}{4}t}\phi_c(x-ct)$ of \eqref{eqs:gDNLS} is stable
if for any $\varepsilon >0$ there exists $\delta >0$ such that if $u_0(x)\in U_\delta(\phi_c)$,
then the solution $u(t,x)$ of \eqref{eqs:gDNLS} with $u(0,x)=u_0(x)$ exists for all $t\in \R$,
and $u(t,x)\in U_\varepsilon(\phi_c)$ for all $t\in \R$.
Otherwise, $e^{i\frac{c^2}{4}t}\phi_c(x-ct)$ is said to be unstable.
\end{defn}

\begin{thm}\label{thm:mainthm}
Let $\sigma\in (1,2)$, then the solitary wave solution $e^{i\frac{c^2}{4}t}\phi_c(x-ct)$ of \eqref{eqs:gDNLS} is unstable.
\end{thm}

As described in Remark \ref{rem:1}, the new feature in the endpoint case is the ``zero mass" properties, which are related to both $\phi_{\omega,c}$ and the functional $S_{\omega,c}$ when $c=2\sqrt\omega$. This new feature brings obstacles in the study of the stability theory. It is worth to noting that the direction of neither $\partial_\omega\phi_{\omega,c}$ nor $\partial_c\phi_{\omega,c}$ makes sense when $c\to 2\sqrt{\omega}$. Especially,
 $\partial_\omega P(\phi_{\omega,c})$ and $\partial_c P(\phi_{\omega,c})$ go to infinity when $c\to 2\sqrt{\omega}$. This makes it impossible to handle this problem in the same way as the  non-endpoint case. Furthermore, in the endpoint case, a two-parameter family of solitary wave solutions ($\omega$ and $c$) degenerates into only one parameter family of solitary wave solutions. This causes the absence of a nature definition to the negative direction which is orthogonal with both $M'(\phi_c)$ and $P'(\phi_c)$.

Our argument is based on \cite{NOW16}, in which the authors constructed an auxiliary function and used the cut-off trick to define the negative direction. However,  the argument in \cite{NOW16} does not work for all $\sigma\in(1,2)$, but only applies when $\sigma$ is close to 1. To overcome the difficulty, we construct a new auxiliary function to solve the problem for any $\sigma\in (1,2)$.

This paper is organized as follows. In Section 2, we give the definitions of some important functionals and some useful lemmas. In Section 3, we construct the negative direction. In Section 4, we prove the Theorem 1.

\section{Preliminaries}

\subsection{Notations}
We use $A\lesssim B$ to denote an estimate of the form
$A\leq CB$ for some constant $C>0$. Also, we write O(A) to indicate any quantity A such that $|A|\lesssim B $.
And we denote $\langle x\rangle=\sqrt{1+x^2}$ .

For a function $f(x)$, its $L^{q}$-norm $\|f\|_{L^q}=\Big(\displaystyle\int_{\mathbb{R}} |f(x)|^{q}dx\Big)^{\frac{1}{q}}$
and its $H^1$-norm $\|f\|_{H^1}=(\|f\|^2_{L^2}+\|\partial_x f\|^2_{L^2})^{\frac{1}{2}}$.
For $u,v\in L^2(\mathbb{R})=L^2(\mathbb{R,C})$, we define
$$(u,v)=\mbox{Re}\int_{\mathbb{R}}u(x)\overline{v(x)}\,dx$$
and regard $L^2(\mathbb{R})$ as a real Hilbert space.

From the definitions of $E$, $P$ and $M$, we have
\begin{align}
E'(u)=&-\partial_x^2u-i|u|^{2\sigma}\partial_xu,\notag\\
P'(u)=&i\partial_xu,\label{P'}\\
M'(u)=&u\label{M'}.
\end{align}
Let
\begin{align*}
S_c(u)&=E(u)+cP(u)+\frac{c^2}{4}M(u),\\
K_c(u)&=\langle S'_c(u),u\rangle.
\end{align*}
Then, we have
\begin{align}\label{S'c}
S'_c(u)&=E'(u)+cP'(u)+\frac{c^2}{4}M'(u)\notag\\
&=-\partial_x^2u-i|u|^{2\sigma}\partial_x u+ci\partial_xu+\frac{c^2}{4}u,
\end{align}
and
\begin{align*}
K_c(u)=\| \partial_x u\|_{L^2}^2-( i|u|^{2\sigma}\partial_x u,u) +c (i\partial_xu,u)+\frac{c^2}{4}\| u\|_{L^2}^2.
\end{align*}
For the solution $\phi_c$ to \eqref{Elliptic-comp}, we have
\begin{align*}
S'_c(\phi_c)=0,
\end{align*}
and thus $K_c(\phi_c)=0$. Moreover, by \eqref{S'c}, we obtain
\begin{align}
S''_c(\phi_c)f=&-\partial_x^2f+ci\partial_xf+\frac{c^2}{4}f-i\sigma|\phi_c|^{2\sigma-2}\overline{\phi_c}\partial_x \phi_c f\notag\\
&-i\sigma|\phi_c|^{2\sigma-2}\phi_c\partial_x \phi_c \overline f
-i |\phi_c|^{2\sigma}\partial_x f .\label{S''c}
\end{align}

\subsection{Useful Lemmas}
In this section, we prove some useful lemmas. 
\begin{lem}\label{14.49}
	$S''_c(\phi_c)$ is self-adjoint, that is, for any $f,g\in\H1$,
	\begin{align*}
	\langle S''_c(\phi_c)f,g\rangle=\langle S''_c(\phi_c)g,f\rangle.
	\end{align*}
	Moreover,
	\begin{align}\label{S'M'P'}
	S''_c(\phi_{c})\partial_c\phi_c=-\frac{c}{2}M'(\phi_c)-P'(\phi_c).
	\end{align}
\end{lem}
\begin{proof}
	Note that
	$$
	\partial_t\partial_sS_c(\phi_c+sg+tf)=\partial_s\partial_tS_c(\phi_c+sg+tf).
	$$
	Then, taking $t=s=0$, we get the first formula in the lemma .
	
	From $S'_c(\phi_{c})=0$, and differentiating it with respect to $c$, we know that
	\begin{align*}
	S''_c(\phi_{c})\partial_c\phi_c=-\frac{c}{2}M'(\phi_c)-P'(\phi_c).
	\end{align*}
	This finishes the proof.
\end{proof}

\begin{lem}
	Let $\sigma\in(1,2)$, then we have
	\begin{align}\label{PM}
	P(\phi_c)=\frac c2(1-\sigma)M(\phi_c),
	\end{align}
	and
	\begin{align}\label{phiM}
	\|\phi_c\|_{L^{2\sigma+2}}^{2\sigma+2}=2c(\sigma+1)(2-\sigma)M(\phi_c).
	\end{align}
	Moreover,
	\begin{align}\label{PcMc}
	\partial_cP(\phi_c)=\frac c2\partial_cM(\phi_c).
	\end{align}
\end{lem}
\begin{proof}
	First, we use the definiton of $M$  and the explicit formula \eqref{0.2} to derive
	\begin{align}\label{M}
	M(\phi_c)&=\frac{1}{2}\|\phi_c\|_{L^2}^2=\frac{1}{2}\|\varphi_c\|_{L^2}^2\notag\\
	&=\frac{1}{2}\int_{\mathbb{R}}\bigg(\frac{2c(\sigma+1)}{\sigma^2(cx)^2+1}\bigg)^{\frac{1}{\sigma}}dx\notag\\
	&=\frac{1}{2}\sigma^{-1}(2\sigma+2)^{\frac{1}{\sigma}}A_\sigma\, c^{\frac{1}{\sigma}-1},
	\end{align}
	where $A_\sigma=\int_{\mathbb{R}}(x^2+1)^{-\frac{1}{\sigma}}dx>0$.
	
	Next, by \eqref{0.1}, we have
	\begin{align}\label{AAB}
	\partial_x\phi_c(x)=e^{i\frac{c}{2}x-\frac{i}{2(\sigma+1)}\int_{-\infty}^{x}\varphi_{c}^{2\sigma}(y)dy}\big[\big(i\frac{c}{2}-\frac{i}{2\sigma+2}\varphi_c^{2\sigma}\big)\varphi_c+\partial_x\varphi_c \big].
	\end{align}
	Then, combining with the definition of $P$ yields
	\begin{align}\label{P}
	P(\phi_c)&=-\frac{c}{2}M(\phi_c)+\frac{1}{2(2\sigma+2)}\|\varphi_c\|_{L^{2\sigma+2}}^{2\sigma+2}\notag\\
	&=-\frac{c}{2}M(\phi_c)+\frac{1}{2(2\sigma+2)}(2\sigma+2)^{\frac 1\sigma+1}c^{\frac 1\sigma+1}\int_{\mathbb{R}}\bigg(\frac{1}{\sigma^2(cx)^2+1}\bigg)^{\frac 1\sigma+1}dx \notag\\
	&=-\frac{c}{2}M(\phi_c)+\frac{1}{2}\sigma^{-1}(2\sigma+2)^{\frac{1}{\sigma}}B_\sigma\,c^{\frac{1}{\sigma}},
	\end{align}
	where $B_\sigma=\int_{\mathbb{R}}(x^2+1)^{-\frac{1}{\sigma}-1}dx$.
	
	The fundamental observation is that
	\begin{align}\label{AAA}
	\frac{d [x(x^2+1)^{-\frac{1}{\sigma}}]}{dx}&=(1-\frac{2}{\sigma})(x^2+1)^{-\frac{1}{\sigma}}+\frac{2}{\sigma}(x^2+1)^{-\frac{1}{\sigma}-1}.
	\end{align}
	Integration of \eqref{AAA} with $x$ for $\sigma\in(1,2)$ yields
	$$\frac{2}{\sigma} \int_{\mathbb{R}}(x^2+1)^{-\frac{1}{\sigma}-1}dx=(\frac{2}{\sigma}-1)\int_{\mathbb{R}}(x^2+1)^{-\frac{1}{\sigma}}dx.
	$$
	That is
	\begin{align}\label{B}
	B_\sigma=(1-\frac{\sigma}{2})A_\sigma.
	\end{align}
	Together with \eqref{M}, \eqref{P} and \eqref{B}, we have
	\begin{align*}
	P(\phi_c)=-\frac{c}{2}M(\phi_c)+c(1-\frac{\sigma}{2})M(\phi_c)
	=\frac c2(1-\sigma)M(\phi_c).
	\end{align*}
	Moreover, from \eqref{P}, we have
	\begin{align*}
	\|\phi_c\|_{L^{2\sigma+2}}^{2\sigma+2}&=\|\varphi_c\|_{L^{2\sigma+2}}^{2\sigma+2}=4(\sigma+1)\big[\frac c2M(\phi_c)+P(\phi_c)\big]\\
	&=2c(\sigma+1)(2-\sigma)M(\phi_c).
	\end{align*}
	On the other hand, differentiating \eqref{M} with respect to $c$, we have
	\begin{align*}
	\partial_cM(\phi_c)&=\frac{1}{2}\sigma^{-1}(2\sigma+2)^{\frac{1}{\sigma}}A_\sigma\, c^{\frac{1}{\sigma}-2}(\frac 1\sigma-1)\\
	&=c^{-1}(\frac 1\sigma-1)M(\phi_c).
	\end{align*}
	That is
	\begin{align}\label{M'cM}
	M(\phi_c)=c\sigma(1-\sigma)^{-1}\partial_cM(\phi_c).
	\end{align}
	Finally, differentiating \eqref{PM} with respect to $c$ and together with \eqref{M'cM} yields
	\begin{align*}
	\partial_cP(\phi_c)&=\frac 12(1-\sigma)M(\phi_c)+\frac c2(1-\sigma)\partial_cM(\phi_c)\\
	&=\frac 12(1-\sigma)c\sigma(1-\sigma)^{-1}\partial_cM(\phi_c)+\frac c2(1-\sigma)\partial_cM(\phi_c)\\
	&=\frac c2\partial_cM(\phi_c).
	\end{align*}
	This completes the proof.
\end{proof}

\begin{lem}\label{S''-SomeV}
	Let $\sigma\in(1,2)$, then
	\begin{align*}
	\langle S''_c(\phi_c)\partial_c\phi_{c},\partial_c\phi_{c}\rangle>0.
	\end{align*}
\end{lem}
\begin{proof}
	Using \eqref{S'M'P'}, we get
	\begin{align}\label{S-M-P}
	\langle S''_c(\phi_c)\partial_c\phi_{c},\partial_c\phi_{c}\rangle=-\frac{c}{2}\partial_cM(\phi_c)-\partial_cP(\phi_c).
	\end{align}
	From \eqref{M} and \eqref{P}, we have
	\begin{align*}
	-\frac{c}{2}\partial_cM(\phi_c)-\partial_cP(\phi_c)
	&=-\frac{c}{2}\partial_cM(\phi_c)+\frac{c}{2}\partial_cM(\phi_c)+\frac 12 M(\phi_c)
	-\frac{1}{2}\sigma^{-2}(2\sigma+2)^{\frac{1}{\sigma}}B_\sigma\,c^{\frac{1}{\sigma}-1}\notag\\
	&=\frac{1}{4}\sigma^{-1}(2\sigma+2)^{\frac{1}{\sigma}}A_\sigma\, c^{\frac{1}{\sigma}-1}
	-\frac{1}{2}\sigma^{-2}(2\sigma+2)^{\frac{1}{\sigma}}B_\sigma\,c^{\frac{1}{\sigma}-1}\notag\\
	&=\frac{1}{4}\sigma^{-2}(2\sigma+2)^{\frac{1}{\sigma}}\left(\sigma A_\sigma-2B_\sigma\right)c^{\frac{1}{\sigma}-1}.
	\end{align*}
	Combining with \eqref{S-M-P}, \eqref{B} and \eqref{M}, we have
	\begin{align*}
	\langle S''_c(\phi_c)\partial_c\phi_{c},\partial_c\phi_{c}\rangle
	&=\frac{1}{4}\sigma^{-2}(2\sigma+2)^{\frac{1}{\sigma}}\big[\sigma A_\sigma-2(1-\frac \sigma 2)B_\sigma\big]\\
	&=\frac{1}{4}\sigma^{-2}(2\sigma+2)^{\frac{1}{\sigma}}2(\sigma-1)A_\sigma>0.
	\end{align*}
	This completes the proof.
\end{proof}

\begin{lem}\label{MaxC}
	Let $\sigma\in(1,2)$, then
	$$\langle S''_c(\phi_c)\phi_c,\phi_c\rangle <0.$$
\end{lem}
\begin{proof}
	From \eqref{S''c} and \eqref{Elliptic-comp}, we have
	\begin{align*}
	S''_{ c}(\phi_{c})\phi_{c}
	=&-\partial_x^2\phi_{c}+\frac{c^2}{4}\phi_{c}+ic\partial_x\phi_{c}-i\sigma|\phi_{c}|^{2\sigma-2}|\phi_{c}|^2\,\partial_x\phi_{c}
	-i\sigma|\phi_{c}|^{2\sigma-2}|\phi_{c}|^2\partial_x\phi_{c}-i|\phi_{c}|^{2\sigma}\partial_x\phi_{c}\\
	=&-\partial_x^2\phi_{c}-(2\sigma+1)i|\phi_{c}|^{2\sigma}\partial_x\phi_{c}+\omega\phi_{c}+ic\partial_x\phi_{c}\\
	=&-2\sigma i |\phi_{c}|^{2\sigma}\partial_x\phi_{c}.
	\end{align*}
	Hence, we obtain
	\begin{align*}
	\langle S''_c(\phi_c)\phi_c,\phi_c\rangle=( -2\sigma i|\phi_c|^{2\sigma}\partial_x\phi_c,\phi_c)
	=-2\sigma\mbox{Im}\int_{\mathbb{R}}|\phi_c|^{2\sigma}\phi_c\,\overline{\partial_x\phi_c}\,dx.
	\end{align*}
	Taking product with $\overline{x\partial_x\phi_c}$ and $\overline{\phi_c}$ in \eqref{Elliptic-comp} respectively, and integrating, we obtain
	\begin{align}\label{22}
	\|\partial_x\phi_c\|_{L^2}^2=\frac{c^2}{4}\|\phi_c\|_{L^2}^2,
	\end{align}
	and
	$$\|\partial_x\phi_c\|_{L^2}^2+\frac{c^2}{4}\|\phi_c\|_{L^2}^2+c\mbox{Im}\int_\R \phi_c\,\overline{\partial_x\phi_c}\,dx-\mbox{Im}\int_\R |\phi_c|^{2\sigma}\phi_c\,\overline{\partial_x\phi_c}\,dx=0.$$
	We collect the above computations and obtain 
	\begin{align}\label{18}
	\mbox{Im}\int_{\mathbb{R}}|\phi_c|^{2\sigma}\phi_c\,\overline{\partial_x\phi_c}\,dx=c^2M(\phi_c)+2cP(\phi_c).
	\end{align}
	Thus, by \eqref{PM}, \eqref{18} and $\sigma\in(1,2)$, we have
	\begin{align*}
	\langle S''_c(\phi_c)\phi_c,\phi_c\rangle
	=&-2\sigma\big[c^2M(\phi_c)+c^2(1-\sigma)M(\phi_c)\big]\\
	=&-2\sigma(2-\sigma)c^2M(\phi_c)<0.
	\end{align*}
	This completes the proof.
\end{proof}

\subsection{Variational characterization}
Next, we consider the following standard minimization problem:
\begin{align}\label{min}
\mu(c)=\inf\{S_c(u):u\in\H1\setminus\{0\},K_c(u)=0\}.
\end{align}
Let $\mathscr{M}_c$ be the set of all minimizations for \eqref{min}, i.e.
$$
\mathscr{M}_c=\{\phi \in \H1\setminus \{0\}:S_c(\phi)=\mu(c),K_c(\phi)=0\}.
$$
Let $\mathscr{G}_c$ be the set of all critical points of $S_c$, then
$$
\mathscr{G}_c=\{\phi\in\H1\setminus\{0\}:S'_c(\phi)=0\}.
$$

The main result of this subsection is following. Since it can be proved by the standard variational argument (see for examples \cite{CoOh-06-DNLS, Soonsik-W-2014, LiSiSu1}, in particular, see \cite{Soonsik-W-2014} for the ``zero mass" case), we omit the details of the proof here.
\begin{lem}\label{relation}
	$\mathscr{G}_c=\{e^{i\theta}\phi_c(\cdot-y):(\theta,y)\in\mathbb{R}^2\}$, and $\mathscr{M}_c=\mathscr{G}_c$. In particular, if $v\in \H1$ satisfies $K_c(v)=0$ and $ v\neq 0$, then $S_c(\phi_c)\leq S_c(v)$.
\end{lem}

\section{Negative direction and modulation} \label{sec:3}

For $R>0$, let $\chi_R(x)=\chi(\frac{x}R)$,
where $\chi\in C^{\infty}(\mathbb{R})$, such that $\chi(x)=1$ when $|x|\le 1$; $\chi(x)=0$ when $|x|\ge 2$.
Because $\partial_c\phi_c$ does not belong to $L^2(\mathbb{R})$, the localization technique is employed here,
as will be seen in the proof of the following lemma. 

\begin{prop}\label{ND}
	There exist $\mu$, $\nu$ and $R$
	such that for the function $ \psi=\phi_c+\mu\chi_R\partial_c\phi_c+\nu i\partial_x\phi_c$,
	the following properties hold:
	\begin{itemize}
		\item[(1)]$ \psi\in H^1(\R)$;
		\item[(2)]$\langle P'(\phi_c),\psi\rangle=\langle M'(\phi_c),\psi\rangle=0$;
		\item[(3)]$\langle S''_c(\phi_c)\psi,\psi\rangle <0$.
	\end{itemize}
\end{prop}
\begin{proof}
	(1) Since $\phi_c\in H^1(\R)$ and $\partial_x\phi_c\in H^1(\R)$, we just need to verify that $\chi_R\partial_c\phi_c\in H^1(\R)$.
	From \eqref{0.1}, we have
	\begin{align}\label{10.18}
	\partial_c\phi_c=e^{i\frac{c}{2}x-\frac{i}{2(\sigma+1)}\int_{-\infty}^{x}{\varphi_c(y)^{2\sigma}dy}}\Big(\frac i2 x\varphi_c-\frac {i\sigma}{\sigma+1}\varphi_c\int_{-\infty}^x\partial_c\varphi_c\varphi_c^{2\sigma-1}dy+\partial_c\varphi_c\Big).
	\end{align}
	By \eqref{10.21}(see Appendix), we know that 
	\begin{align*}
	&|\partial_c\phi_c|\lesssim \langle x\rangle^{1-\frac{1}{\sigma}},\quad\mbox{and}\quad  |\partial_x\partial_c\phi_c|\lesssim \langle x\rangle^{1-\frac{1}{\sigma}}.
	\end{align*}
	Since $\chi_R(x)$ is smooth cutoff function, we have   $\chi_R\partial_c\phi_c\in H^1(\R)$.
	
	(2) It is sufficient to find $\mu$, $\nu$ such that
	\begin{equation*}
	\left\{ \aligned
	&\langle P'(\phi_c),\phi_c+\mu\chi_R\partial_c\phi_c+\nu i\partial_x\phi_c\rangle=0,\\
	&\langle M'(\phi_c),\phi_c+\mu\chi_R\partial_c\phi_c+\nu i\partial_x\phi_c\rangle=0.
	\endaligned
	\right.
	\end{equation*}
	Together with \eqref{P'} and \eqref{M'}, we obtain
	\begin{align}
	\mu = -\frac{4P(\phi_c)^2-2M(\phi_c)\|\partial_x\phi_c\|_{L^2}^2}
	{2P(\phi_c)\cdot\frac{1}{2}\partial_c\mbox{Im}\int_{\R}\chi_R\phi_c\overline{\partial_x\phi_c}dx
		-\|\partial_x\phi_c\|_{L^2}^2\cdot\frac{1}{2}\partial_c\int_{\R}\chi_R|\phi_c|^2dx},\label{zhi1}
	\end{align}
	and
	\begin{align}
	\nu= \frac{2P(\phi_c)\cdot\frac{1}{2}\partial_c\int_{\R}\chi_R|\phi_c|^2dx
		-2M(\phi_c)\cdot\frac{1}{2}\partial_c\mbox{Im}\int_{\R}\chi_R\phi_c\overline{\partial_x\phi_c}dx}
	{2P(\phi_c)\cdot\frac{1}{2}\partial_c\mbox{Im}\int_{\R}\chi_R\phi_c\overline{\partial_x\phi_c}dx
		-\|\partial_x\phi_c\|_{L^2}^2\cdot\frac{1}{2}\partial_x\int_{\R}\chi_R|\phi_c|^2dx}.\label{zhi2}
	\end{align}
	Inserting \eqref{PM}, \eqref{PcMc}, \eqref{22} into \eqref{zhi1} and \eqref{zhi2} and using Lemmas A.1 and A.2 yields 
	\begin{align*}
	\mu=\frac{-2(2-\sigma)M(\phi_c)}{\partial_c M(\phi_c)+O(R^{-\frac{2}{\sigma}+1})},
	\end{align*}
	and
	\begin{align}\label{nu}
	\nu=\frac{2}{c}+O(R^{-\frac{2}{\sigma}+1}).
	\end{align}
	
	(3) According to Lemma \ref{14.49} and the selection of $\psi$, we have
	\begin{align*}
	\langle S''_c(\phi_c)\phi_c,\phi_c\rangle
	=\langle S''_c(\phi_c)(\psi-\mu\chi_R\partial_c\phi_c-\nu i\partial_x\phi_c),\psi-\mu\chi_R\partial_c\phi_c-\nu i\partial_x\phi_c \rangle.
	\end{align*}
	By the self-adjoint of $S''_c(\phi_c)$ and a direct expansion, we obtain
	\begin{align}
	\langle S''_c(\phi_c)\phi_c,\phi_c\rangle
	&=\langle S''_c(\phi_c)\psi,\psi\rangle-2\mu\langle S''_c(\phi_c)\chi_R\partial_c\phi_c,\psi\rangle
	-2\nu \langle S''_c(\phi_c)\psi,i\partial_x\phi_c\rangle\notag\\
	+&\mu^2\langle S''_c(\phi_c)\chi_R\partial_c\phi_c,\chi_R\partial_c\phi_c\rangle+2\mu\nu\langle S''_c(\phi_c)\chi_R\partial_c\phi_c,i\partial_x\phi_c\rangle\notag\\
	+&\nu^2\langle S''_c(\phi_c)i\partial_x\phi_c,i\partial_x\phi_c\rangle.\label{15.55}
	\end{align}
	Using $\psi=\phi_c+\mu\chi_R\partial_c\phi_c+\nu i\partial_x\phi_c$, we have
	\begin{align}
	\langle S''_c(\phi_c)\psi,i\partial_x\phi_c\rangle=&\langle S''_c(\phi_c)(\phi_c+\mu\chi_R\partial_c\phi_c+\nu i\partial_x\phi_c),i\partial_x\phi_c\rangle\notag\\
	=&\langle S''_c(\phi_c)\phi_c,i\partial_x\phi_c\rangle+\mu\langle S''_c(\phi_c)\chi_R\partial_c\phi_c,i\partial_x\phi_c\rangle\notag\\
	&+\nu \langle S''_c(\phi_c)i\partial_x\phi_c,i\partial_x\phi_c\rangle.\label{15.54}
	\end{align}
	Together with \eqref{15.55} and \eqref{15.54}, we get
	\begin{align}\label{16}
	\langle S''_c(\phi_c)\phi_c,\phi_c\rangle=&\langle S''_c(\phi_c)\psi,\psi\rangle-2\mu\langle S''_c(\phi_c)\chi_R\partial_c\phi_c,\psi\rangle-2\nu\langle S''_c(\phi_c)\phi_c,i\partial_x\phi_c\rangle\notag\\
	&+\mu^2\langle S''_c(\phi_c)\chi_R\partial_c\phi_c,\chi_R\partial_c\phi_c\rangle
	-\nu^2\langle S''_c(\phi_c)i\partial_x\phi_c,i\partial_x\phi_c\rangle.
	\end{align}
	Combining with \eqref{S'M'P'} and the conclusion (2), we have
	$$
	\langle S''_c(\phi_c)\partial_c\phi_c,\psi\rangle=0.
	$$
	Then, we know
	\begin{align}\label{3.18}
	\langle S''_c(\phi_c)\chi_R\partial_c\phi_c,\psi\rangle
	=& -\langle S''_c(\phi_c)(1-\chi_R)\partial_c\phi_c,\psi\rangle.
	\end{align}
	Inserting \eqref{3.18} into \eqref{16} yields
	\begin{align}\label{17}
	\langle S''_c(\phi_c)\phi_c,\phi_c\rangle
	=&\langle S''_c(\phi_c)\psi,\psi\rangle+2\mu\langle S''_c(\phi_c)(1-\chi_R)\partial_c\phi_c,\psi\rangle-2\nu\langle S''_c(\phi_c)\phi_c,i\partial_x\phi_c\rangle\notag\\
	&-\nu^2\langle S''_c(\phi_c)i\partial_x\phi_c,i\partial_x\phi_c\rangle
	+\mu^2\langle S''_c(\phi_c)(1-\chi_R)\partial_c\phi_c,(1-\chi_R)\partial_c\phi_c\rangle\notag\\
	&-2\mu^2\langle S''_c(\phi_c)(1-\chi_R)\partial_c\phi_c,\partial_c\phi_c\rangle+\mu^2\langle S''_c(\phi_c)\partial_c\phi_c,\partial_c\phi_c\rangle.
	\end{align}
	From Lemma A.3, we have
	\begin{align}\label{1}
	|\langle S''_c(\phi_c)(1-\chi_R)\partial_c\phi_c,\partial_c\phi_c\rangle|
	\lesssim &\int\left(1-\chi_{\frac{R}{2}}\right)\langle x\rangle^{-1-\frac{1}{\sigma}}\langle x\rangle^{-1-\frac{1}{\sigma}}dx\notag\\
	=&O(R^{-\frac{2}{\sigma}+1}),
	\end{align}
	and
	\begin{align}\label{2}
	|\langle S''_c(\phi_c)(1-\chi_R)\partial_c\phi_c,(1-\chi_R)\partial_c\phi_c\rangle|=O(R^{-\frac{2}{\sigma}+1}).
	\end{align}
	Note that $|\psi|\lesssim \langle x\rangle^{1-\frac{1}{\sigma}}$, we get
	\begin{align}
	|\langle S''_c(\phi_c)(1-\chi_R)\partial_c\phi_c,\psi\rangle|
	&\lesssim \int\left(1-\chi_{\frac{R}{2}}\right)\langle x\rangle^{-1-\frac{1}{\sigma}}|\psi|dx\notag\\
	&=O(R^{-\frac{2}{\sigma}+1}).\label{3}
	\end{align}
	Hence, inserting the estimates in \eqref{1}--\eqref{3} into \eqref{17}, and using \eqref{nu}, we get
	\begin{align}\label{17.1}
	\langle S''_c(\phi_c)\phi_c,\phi_c\rangle
	=&\langle S''_c(\phi_c)\psi,\psi\rangle-\frac{4}{c}\langle S''_c(\phi_c)\phi_c,i\partial_x\phi_c\rangle-\nu^2\langle S''_c(\phi_c)i\partial_x\phi_c,i\partial_x\phi_c\rangle\notag\\
	&
	\quad +\mu^2\langle S''_c(\phi_c)\partial_c\phi_c,\partial_c\phi_c\rangle+O(R^{-\frac{2}{\sigma}+1}).
	\end{align}
	Now we need the following lemma.
	\begin{lem}\label{AAAA}
		It holds that
		\begin{align*}
	\langle S''_c(\phi_c)i\partial_x\phi_c,i\partial_x\phi_c\rangle<0,\quad \quad\mbox{and }
	\langle S''_c(\phi_c)i\partial_x\phi_c,\phi_c\rangle<0.
		\end{align*}
	\end{lem}
	\begin{proof}
		From \eqref{Elliptic-comp} and \eqref{S''c}, we have
		$$ S''_c(\phi_c)i\partial_x\phi_c=-\frac{c^2}{2}\sigma|\phi_c|^{2\sigma}\phi_c.$$
		Therefore,
		\begin{align*}
		\langle S''_c(\phi_c)i\partial_x\phi_c,i\partial_x\phi_c\rangle=&-\frac{c^2}{2}\sigma( |\phi_c|^{2\sigma}\phi_c,i\partial_x\phi_c)\\
		=&-\frac{c^2}{2}\sigma\mbox{Im}\int_\mathbb{R}|\phi_c|^{2\sigma}\phi_c\overline{\partial_x\phi_c}dx.
		\end{align*}
		From \eqref{18} and \eqref{PM}, we get
		\begin{align*}
		\langle S''_c(\phi_c)i\partial_x\phi_c,i\partial_x\phi_c\rangle=&-\frac{c^2}{2}\sigma\left[c^2M(\phi_c)+2cP(\phi_c)\right]\notag\\
		=&-\frac{c^4}{2}(2-\sigma)\sigma M(\phi_c)<0.
		\end{align*}
		Similarly, we have
		\begin{align*}
		\langle S''_c(\phi_c)i\partial_x\phi_c,\phi_c\rangle
		=&-\frac{c^2}{2}\sigma\|\phi_c\|_{L^{2\sigma+2}}^{2\sigma+2}.
		\end{align*}
		From \eqref{phiM} ,  we obtain
		\begin{align*}
		\langle S''_c(\phi_c)i\partial_x\phi_c,\phi_c\rangle =-(\sigma+1)(2-\sigma)\sigma c^3M(\phi_c)<0.
		\end{align*}
		This proves the lemma.
	\end{proof}
	
	Combining with \eqref{17.1} and Lemma \ref{AAAA}, we have
	\begin{align*}
	\langle S''_c(\phi_c)\psi,\psi\rangle
	=&\langle S''_c(\phi_c)\phi_c,\phi_c\rangle+\frac{4}{c}\langle S''_c(\phi_c)\phi_c,i\partial_x\phi_c\rangle+\nu^2\langle S''_c(\phi_c)i\partial_x\phi_c,i\partial_x\phi_c\rangle\notag\\
	&-\mu^2\langle S''_c(\phi_c)\partial_c\phi_c,\partial_c\phi_c\rangle+O(R^{-\frac{2}{\sigma}+1})\\
	<&\langle S''_c(\phi_c)\phi_c,\phi_c\rangle-\mu^2\langle S''_c(\phi_c)\partial_c\phi_c,\partial_c\phi_c\rangle+O(R^{-\frac{2}{\sigma}+1}).
	\end{align*}
	From Lemma \ref{S''-SomeV} and Lemma \ref{MaxC}, we note that the first and the second terms in the right-hand side are negative. Hence, choosing $R$ large enough, we obtain
	$$\langle S''_c(\phi_c)\psi,\psi\rangle<\langle S''_c(\phi_c)\phi_c,\phi_c\rangle<0.$$
	This concludes the proof of Proposition \ref{ND}.
\end{proof}

\begin{lem}\label{cor: Sc}
	There exists a constant $\beta_0>0$ such that
	$$S_c(\phi_c+\beta\psi)<S_c(\phi_c),$$
	for all $\beta\in (-\beta_0,0)\cup(0,\beta_0).$
\end{lem}

\begin{proof}
	By Taylor's expansion, for $\beta\in\R$, we have
	\begin{align*}
	S_c(\phi_c+\beta\psi)=&S_c(\phi_c)+\beta\langle S'_c(\phi_c),\psi\rangle
	+\frac{1}{2}\beta^2\langle S''_c(\phi_c)\psi,\psi\rangle+o(\beta^2)\\
	=&S_c(\phi_c)+\frac{1}{2}\beta^2\langle S''_c(\phi_c)\psi,\psi\rangle+o(\beta^2).
	\end{align*}
	Since $\langle S''_c(\phi_c)\psi,\psi\rangle <0$, there exists a constant $\beta_0>0$, such that
	for any $\beta\in (-\beta_0,0)\cup(0,\beta_0)$, we have
	$$S_c(\phi_c+\beta\psi)<S_c(\phi_c).$$
	This finishes the proof.
\end{proof}

We denote $\T=\R/2\pi\Z$. Then we can get the following proposition.

\begin{prop}\label{pro}
	Suppose $u\in U_{\varepsilon_0}(\phi_c)$, then exist $\theta=\theta(u)$, $y=y(u)$, such that
	
	{\rm(1)} $\langle u,ie^{i\theta}\phi_c(\cdot-y)\rangle=0,\ \langle u,e^{i\theta}\partial_x\phi_c(\cdot-y)\rangle=0$;
	
	{\rm(2)} $\|\partial_u\theta\|_{\H1}\leq C$ and $\|\partial_u y\|_{\H1}\leq C$ for any $u\in U_{\varepsilon_0}(\phi_c)$;   
	
	{\rm(3)} $\theta(e^{i\theta_0}u(\cdot-y_0))=\theta+\theta_0$, $y(e^{i\theta_0}u(\cdot-y_0))=y+y_0$ for any $u\in U_{\varepsilon_0}(\phi_c)$ and $\theta_0\in\T$, $y_0\in\R$.
	
\end{prop}

\begin{proof}
	Denote $$F_1(\theta,y;u)=\langle u,ie^{i\theta}\phi_c(\cdot-y)\rangle,\ F_2(\theta,y;u)=\langle u,e^{i\theta}\partial_x\phi_c(\cdot-y)\rangle.$$
	Then $F_1(0,0;\phi_c)=F_2(0,0;\phi_c)=0$.
	
	According to the definitions of $F_1$ and $F_2$, we have
	\begin{equation*}
	\partial_{\theta}F(u,\theta)=
	\left(
	\begin{array}{cc}
	\partial_{\theta}F_1(\theta,y;u) &\partial_{y}F_1(\theta,y;u) \\
	\partial_{\theta}F_2(\theta,y;u) & \partial_{y}F_2(\theta,y;u) \\
	\end{array}
	\right).
	\end{equation*}
	Moreover, we have 
	\begin{align*}
	\partial_{\theta}F_1\mid_{(0,0;\phi_c)}=-\|\phi_c\|_{L^2}^2,\qquad
	&\partial_{y}F_1\mid_{(0,0;\phi_c)}=-2P(\phi_c),\\
	\partial_{\theta}F_2\mid_{(0,0;\phi_c)}=2P(\phi_c),\qquad
	&\partial_{y}F_2\mid_{(0,0;\phi_c)}=\|\partial_x\phi_c\|_{L^2}^2.
	\end{align*}
	Then, from \eqref{22} and \eqref{PM}, the Jacobian 
	\begin{align*}
	|\partial_{\theta}F(u,\theta)\mid_{(0,0;\phi_c)}=&-\|\phi_c\|_{L^2}^2\|\partial_x\phi_c\|_{L^2}^2+4P^2(\phi_c)\\
	=&-\sigma(2-\sigma)c^2M^2(\phi_c)\ne 0.
	\end{align*}
	Therefore by implicit function theorem, there exist a $\varepsilon_0>0$ and a unique $\C^1$-function
	$\theta=\theta(u)$, $y=y(u)$
	such that for any $u\in U_{\varepsilon_0}(\phi_c)$,
	$$\langle u,ie^{i\theta}\phi_c(\cdot-y)\rangle=0,\ \langle u,e^{i\theta}\partial_x\phi_c(\cdot-y)\rangle=0.$$
	
	Moreover, (2) follows from the implicit function differentiability theorem, and (3) follows from the uniqueness of the implicit functions. 
	
	This concludes the proof of Proposition \ref{pro}.
\end{proof}

\section{proof of Theorem \ref{thm:mainthm}}
We argue for contradiction and suppose that $u\in U_{\varepsilon_0}(\phi_c)$. Moreover, we define
$$A(u)=\langle iu,e^{i\theta}\psi(\cdot-y)\rangle,$$
and
$$
q(u)=iA'(u).
$$
Then, we have
\begin{align}\label{15}
q(u)=e^{i\theta}\psi(\cdot-y)+i\theta_u\langle u,e^{i\theta}\psi(\cdot-y)\rangle+i y_u \langle iu,-e^{i\theta}\partial _x\psi(\cdot-y)\rangle.
\end{align}

\begin{lem}
	For $u\in U_{\varepsilon_0}(\phi_c)$, $q(u)$ is continuous from $U_{\varepsilon_0}(\phi_c)$ to $\H1$ and  $q(\phi_c)=\psi$.
\end{lem}
\begin{proof}
	By Proposition \ref{ND} (2),
	\begin{align*}
	q(\phi_c)&=\psi+(\phi_c,\psi)i\theta_u(\phi_c)+(i\phi_c,-\partial_x\psi)y_u(\phi_c)\\
	&=\psi+(\phi_c,\psi)i\theta_u(\phi_c)+(i\partial_x\phi_c,\psi)y_u(\phi_c)\\
	&=\psi.
	\end{align*}
	Moreover, from the definition \eqref{15} and Proposition \ref{pro} (2), we know that $q(u)$ is continuous from $U_{\varepsilon_0}(\phi_c)$ to $\H1$.
	This proves the lemma.
\end{proof}

Now, we prove Theorem 1.
\begin{proof}
	From \eqref{eqs:gDNLS}, we know $i\partial_tu=E'(u)$, thus
	$$\partial_t A(u)=\langle A'(u), \partial_tu \rangle=\langle iA'(u),E'(u)\rangle.$$
	Since $A\left(e^{i\theta_0}u(\cdot-y_0)\right)=A(u)$, for any $(\theta_0,y_0)\in\R^2$. Differentiating with $\theta_0$ and $y_0$, we have
	\begin{align*}
	\langle iA'(u),M'(u)\rangle=\langle iA'(u),P'(u)\rangle=0.
	\end{align*}
	Note that $q(u)=i A'(u)$, then using the identities above, we have
	\begin{align*}
	\partial_tA(u(t))&=\langle iA'(u),S'_c(u)\rangle=\langle q(u),S'_c(u)\rangle\\
	&=\frac{1}{\lambda}\left[S_c(u+\lambda q(u))-S_c(u)-\lambda^2\int_{0}^{1}(1-s)\langle S''_c\left(\phi_c+s\lambda q(u)\right)q(u),q(u)\rangle ds\right].
	\end{align*}
	
   Next, we denote $\tilde u=e^{-i\theta}u(\cdot+y)$. Combinig with $u\in U_{\varepsilon_0}(\phi_c)$ and Proposition \ref{pro}, we have 
   $$\|\tilde u-\phi_c\|_{H^1}\le \varepsilon_0.$$
	Then, choosing $\lambda,\ \varepsilon_0$ small enough, and by Proposition \ref{ND}, we get 
	\begin{align*}
	\int_{0}^{1}(1-s)\langle S''_c\left(\phi_c+s\lambda q(u)\right)q(u),q(u)\rangle ds
	&=\int_{0}^{1}(1-s)\langle S''_c\left(\phi_c+s\lambda q(\widetilde{u})\right)q(\widetilde u),q(\widetilde{u})\rangle ds\\
	&=\langle S''(\phi_c)\psi, \psi \rangle+O\left (\lambda+\| q(\widetilde{u})-q(\phi_c)\|_{H^1}\right)\\
	&=\langle S''(\phi_c)\psi, \psi \rangle+O(\lambda+\varepsilon_0)\\
	&<\frac{1}{2}\langle S''(\phi_c)\psi, \psi \rangle\\
	&<0.
	\end{align*}
	Hence, we get
	\begin{align}\label{A'}
	\partial_tA(u(t))>\frac{1}{\lambda}\left[S_c(u+\lambda q(u))-S_c(u)\right].
	\end{align}
	Now we claim that
	\begin{align}
	\langle K'_c(\phi_c),\psi\rangle\neq 0.\label{17.51}
	\end{align}
	To show this, we need the following lemma.
	\begin{lem}\label{v}
		If $v\in\H1$ satisfies $\langle K'_c(\phi_c),v\rangle=0$, then $\langle S''_c(\phi_c)v,v\rangle\geq 0$.
	\end{lem}
	\begin{proof}
		See Lemma 4 in \cite{Ohta-14} for the proof.
	\end{proof}
	By Proposition \ref{ND} (3) and Lemma \ref{v}, we have \eqref{17.51}.
	Then applying the implicit functional theorem, we can  find a $\lambda(u)\in(-\lambda_0,\lambda_0)\setminus\{0\}$, such that for any $u\in U_{\varepsilon_0}(\phi_c)$,
	$$K_c(u+\lambda(u)q(u))=0.$$
	Hence, by Lemma \ref{relation}, we have
	$$
	S_c(u+\lambda(u)q(u))\ge S_c(\phi_c).
	$$
	Without loss of generality, we assume $\lambda(u)>0$.
	We choose
	$$
	u_0=\phi_c+\beta\psi.
	$$
	Then, by the conservation laws, we have
	$S_c(u)=S_c(\phi_c+\beta\psi)$.
	Hence,
	\begin{align*}
	S_c(u+\lambda(u)q(u))-S_c(u)\geq S_c(\phi_c)-S_c(\phi_c+\beta\psi).
	\end{align*}
	From  Lemma \ref{cor: Sc}, we have $S_c(\phi_c)-S_c(\phi_c+\beta\psi)>0$. Thus, by \eqref{A'},
	$$\partial_tA(u(t))\geq \frac{1}{\lambda_0}(S_c(\phi_c)-S_c(u_0))>0.$$
	Therefore, we get that $A(u(t))\rightarrow +\infty$ as $t\rightarrow \infty$. However,
	$$|A(u(t))|\leq\|u\|_{L^2}\|\psi\|_{L^2}\leq C\ \ {\rm for\ any}\ \ t>0.$$
	This is a contradiction. This finishes the proof of Theorem \ref{thm:mainthm}.
\end{proof}

\renewcommand{\theequation}{A.\arabic{equation}}
\setcounter{equation}{0}
\newtheorem{Alem}{Lemma A.}
\section*{Appendix}

In this appendix, we prove the following element lemmas used in Section \ref{sec:3}.
\begin{Alem}\label{gj2}
	Let $R>0$, then
	\begin{align*}
	\frac{1}{2}\partial_c\int_{\R}\chi_R|\phi_c|^2dx=\partial_c M(\phi_c)+O(R^{-\frac{2}{\sigma}+1}).
	\end{align*}
\end{Alem}
\begin{proof}
	From the definition of $M$, we have
	\begin{align*}
	\frac{1}{2} \int_{\R}\chi_R|\phi_c|^2dx=  M(\phi_c)+\frac{1}{2} \int_{\R}\left(\chi_R-1\right)|\phi_c|^2dx.
	\end{align*}
	Together with \eqref{0.1} and \eqref{0.2}, we get
	\begin{align}\label{20}
	\int_{\R}\left(\chi_R-1\right)|\phi_c|^2dx=&\int_{\R}\left(\chi_R-1\right)\varphi_c^2dx\notag\\
	=&\left[2c(\sigma+1)\right]^{\frac{1}{\sigma}}\sigma^{-1}c^{-1}\int_{\R}\Bigl[\chi\left(\frac{x}{\sigma c R}\right)-1\Bigr]\left(\frac{1}{1+x^2}\right)^{\frac{1}{\sigma}}dx\notag\\
	=&c_1(\sigma)c^{\frac{1}{\sigma}-1}\int_{\R}\Bigl[\chi\left(\frac{x}{\sigma c R}\right)-1\Bigr]\left(\frac{1}{1+x^2}\right)^{\frac{1}{\sigma}}dx,
	\end{align}
	where $c_1(\sigma)=\sigma^{-1}\left[2(\sigma+1)\right]^{\frac{1}{\sigma}}$.
	Now, differentiating \eqref{20} with respect to $c$, we have
	\begin{align*}
	\partial_c\int_{\R}\left(\chi_R-1\right)|\phi_c|^2dx
	=&(\frac{1}{\sigma}-1)c_1(\sigma)c^{\frac{1}{\sigma}-2}\int_{\R}\Bigl[\chi\left(\frac{x}{\sigma c R}\right)-1\Bigr]\left(\frac{1}{1+x^2}\right)^{\frac{1}{\sigma}}dx\\
	&\quad -c_1(\sigma)c^{\frac{1}{\sigma}-2}\int_{\R}\chi'\left(\frac{x}{\sigma c R}\right)\frac{x}{\sigma  cR}
	\left(\frac{1}{1+x^2}\right)^{\frac{1}{\sigma}}dx.
	\end{align*}
	Note that
	$$
	\int_{\R}\Bigl[\chi\left(\frac{x}{\sigma c R}\right)-1\Bigr]\left(\frac{1}{1+x^2}\right)^{\frac{1}{\sigma}}dx, \quad
	\int_{\R}\chi'\left(\frac{x}{\sigma c R}\right)\frac{x}{\sigma  cR}
	\left(\frac{1}{1+x^2}\right)^{\frac{1}{\sigma}}dx=O(R^{-\frac{2}{\sigma}+1}),
	$$
	we obtain
	$$
	\partial_c\int_{\R}\left(\chi_R-1\right)|\phi_c|^2dx
	=O(R^{-\frac{2}{\sigma}+1}).
	$$
	This finishes the proof.
\end{proof}
\begin{Alem}\label{gj1}
	Let $R>0$, then \begin{align*}
	\frac{1}{2}\partial_c\mbox{Im}\int_{\R}\chi_R\phi_c\overline{\partial_x\phi_c}dx=\partial_c P(\phi_c)+O(R^{-\frac{2}{\sigma}+1}).
	\end{align*}
\end{Alem}
\begin{proof}
	By the definition of $P$, we have
	\begin{align*}
	\frac{1}{2}\partial_c\mbox{Im}\int_{\R}\chi_R\phi_c\overline{\partial_x\phi_c}dx
	=\partial_cP(\phi_c)+\frac{1}{2}\partial_c\mbox{Im}\int_{\R}(\chi_R-1)\phi_c\overline{\partial_x\phi_c}dx.
	\end{align*}
	From \eqref{0.1} and \eqref{AAB}, we obtain
	\begin{align*}
	\mbox{Im}\int_{\R}(\chi_R-1)\phi_c\overline{\partial_x\phi_c}dx
	=&\mbox{Im}\int_{\R}(\chi_R-1)\varphi_c\left(-\frac{c}{2}i\varphi_c+\frac{i}{2(\sigma+1)}\varphi_c^{2\sigma+1}\right)dx\\
	=&-\frac{c}{2}\int_{\R}(\chi_R-1)\varphi_c^2dx+\frac{1}{2(\sigma+1)}\int_{\R}(\chi_R-1)\varphi_c^{2\sigma+2}dx.
	\end{align*}
	Using \eqref{0.2}, we further write $\mbox{Im}\int_{\R}(\chi_R-1)\phi_c\overline{\partial_x\phi_c}dx$ as
	\begin{align}\label{19}
	&-\frac{c}{2}\big[2c(\sigma+1)\big]^{\frac{1}{\sigma}}\sigma^{-1}c^{-1}\int_{\R}\Bigl[\chi\left(\frac{x}{\sigma c R}\right)-1\Bigr]\left(\frac{1}{1+x^2}\right)^{\frac{1}{\sigma}}dx\notag\\
	&\quad +\frac{1}{2(\sigma+1)}\big[2c(\sigma+1)\big]^{\frac{1}{\sigma}+1}\sigma^{-1}c^{-1}\int_{\R}\Bigl[\chi\left(\frac{x}{\sigma c R}\right)-1\Bigr]\left(\frac{1}{1+x^2}\right)^{\frac{1}{\sigma}+1}dx\notag\\
   &=c_2(\sigma)c^{\frac{1}{\sigma}}\int_{\R}\Bigl[\chi\left(\frac{x}{\sigma c R}\right)-1\Bigr]\left(\frac{1}{1+x^2}\right)^{\frac{1}{\sigma}}dx\notag\\
	&\quad+ c_3(\sigma)c^{\frac{1}{\sigma}}\int_{\R}\Bigl[\chi\left(\frac{x}{\sigma c R}\right)-1\Bigr]\left(\frac{1}{1+x^2}\right)^{\frac{1}{\sigma}+1}dx,
	\end{align}
	where $c_2(\sigma)=-\frac{1}{2}\sigma^{-1}\big[2(\sigma+1)\big]^{\frac{1}{\sigma}}$ and $c_3(\sigma)=\sigma^{-1}\big[2(\sigma+1)\big]^{\frac{1}{\sigma}}$. Differentiating \eqref{19} with respect to $c$, and treating similarly as the proof in the previous lemma, we obtain
	\begin{align*}
	\partial_c\mbox{Im}\int_{\R}(\chi_R-1)\phi_c\overline{\partial_x\phi_c}dx
	=&\frac{1}{\sigma}c^{\frac{1}{\sigma}-1}\biggl[c_2(\sigma)\int_{\R}\Bigl[\chi\left(\frac{x}{\sigma c R}\right)-1\Bigr]\left(\frac{1}{1+x^2}\right)^{\frac{1}{\sigma}}dx\\
	&+ c_3(\sigma)\int_{\R}\Bigl[\chi\left(\frac{x}{\sigma c R}\right)-1\Bigr]\left(\frac{1}{1+x^2}\right)^{\frac{1}{\sigma}+1}dx\biggr]\\
	&-c_2(\sigma)c^{\frac{1}{\sigma}-1}\int_{\R}\chi'\left(\frac{x}{\sigma c R}\right)\frac{x}{\sigma  cR}
	\left(\frac{1}{1+x^2}\right)^{\frac{1}{\sigma}}dx\\
	&-c_3(\sigma)c^{\frac{1}{\sigma}-1}\int_{\R}\chi'\left(\frac{x}{\sigma c R}\right)\frac{x}{\sigma  cR}
	\left(\frac{1}{1+x^2}\right)^{\frac{1}{\sigma}+1}dx\\
	=&O(R^{-\frac{2}{\sigma}+1})+O(R^{-\frac{2}{\sigma}-1})\\
	=&O(R^{-\frac{2}{\sigma}+1}).
	\end{align*}
	This proves the lemma.
\end{proof}

\begin{Alem}\label{lem:S''-Rc}
	Let $R>0$. Then
	\begin{align*}
	|S''_c(\phi_c)(1-\chi_R)\partial_c\phi_c\,(x)|\lesssim \big(1-\chi_{\frac R2}(x)\big)\langle x\rangle^{-1-\frac{1}{\sigma}}.
	\end{align*}
\end{Alem}
\begin{proof}
	According to \eqref{0.2}, we get
	\begin{align}\label{10.19}
	|\varphi_c|\lesssim \langle x\rangle^{-\frac{1}{\sigma}}, \quad |\partial_x\varphi_c|\lesssim \langle x\rangle^{-1-\frac{1}{\sigma}}, \quad\mbox{and}\quad |\partial_{xx}\varphi_c|\lesssim \langle x\rangle^{-2-\frac{1}{\sigma}}.
	\end{align}
	Moreover,
	\begin{align}\label{10.20}
	|\partial_c\varphi_c|\lesssim \langle x\rangle^{-\frac{1}{\sigma}},\quad\mbox{and}\quad |\partial_x \partial_c\varphi_c|\lesssim \langle x\rangle^{-1-\frac{1}{\sigma}}.
	\end{align}
	Combining with \eqref{AAB},  \eqref{10.18}, \eqref{10.19} and \eqref{10.20}, we have
	\begin{align}\label{10.21}
	|\partial_x\phi_c|\lesssim\langle x\rangle^{-\frac{1}{\sigma}},\quad|\partial_c\phi_c|\lesssim \langle x\rangle^{1-\frac{1}{\sigma}},\quad\mbox{and}\quad   |\partial_x\partial_c\phi_c|\lesssim \langle x\rangle^{1-\frac{1}{\sigma}}.
	\end{align}
	For suitable function $f$, we note that
	$$
	\partial_x^2f-ci\partial_xf-\frac{c^2}{4}f=e^{\frac c2ix}\partial_x^2\big(e^{-\frac c2ix}f\big).
	$$
	Together with \eqref{S''c}, we can write $S''_c(\phi_c)(1-\chi_R)\partial_c\phi_c$ as
	\begin{align}\label{S''x}
	&-e^{\frac c2ix}\partial_{xx}\Big[e^{-\frac{i}{2(\sigma+1)}\int_{-\infty}^{x}{\varphi_c(y)^{2\sigma}dy}}(1-\chi_R)\Big(\partial_c\varphi_c+\frac i2 x\varphi_c-\frac {i\sigma}{\sigma+1}\varphi_c\int_{-\infty}^x\partial_c\varphi_c\varphi_c^{2\sigma-1}dy\Big)\Big]\nonumber\\
	&\quad-(1-\chi_R)\Big[i\sigma|\phi_c|^{2\sigma-2}\overline{\phi_c}\partial_x \phi_c\partial_c\phi_c+i\sigma|\phi_c|^{2\sigma-2}\phi_c\partial_x \phi_c \overline {\partial_c\phi_c}+i |\phi_c|^{2\sigma}\partial_x \partial_c\phi_c \Big].
	\end{align}
	We collect the computations \eqref{10.19}--\eqref{10.21} and obtain that every term can be controlled by $\langle x\rangle^{-1-\frac{1}{\sigma}}$ in \eqref{S''x}. Thus, we have
	\begin{align*}
	|S''_c(\phi_c)(1-\chi_R)\partial_c\phi_c\,(x)|\lesssim \langle x\rangle^{-1-\frac{1}{\sigma}}.
	\end{align*}
	Finally, we observe that the support of $S''_c(\phi_c)(1-\chi_R)\partial_c\phi_c$ is included in $[R,+\infty)$.
	
	This concludes the proof of Lemma \ref{lem:S''-Rc}.
\end{proof}


\begin{thebibliography}{99}








\bibitem{BiLi-01-Illposed-DNLS-BO}
Biagioni, H. and  Linares, F., {Ill-posedness for the
derivative Schr\"{o}dinger and generalized Benjamin-Ono equations,}
Trans.\ Amer.\ Math.\ Soc., {353} (9), 3649--3659 (2001).


\bibitem{CoOh-06-DNLS}Colin, M. and Ohta, M., {Stability of solitary waves for derivative nonlinear Schr\"odinger equation},
Ann. I. H. Poincar\'{e}-AN, 23, 753--764 (2006).

\bibitem{CKSTT-01-DNLS}
Colliander, J., Keel, M., Staffilani, G., Takaoka, H. and Tao, T.,
{Global well-posedness result for Schr\"{o}dinger equations with
derivative,} SIAM J.\ Math.\ Anal., { 33} (2), 649--669 (2001).

\bibitem{CKSTT-02-DNLS}
Colliander, J., Keel, M., Staffilani, G.,  Takaoka, H. and Tao, T.,
{A refined global well-posedness result for Schr\"{o}dinger
equations with derivatives,} SIAM J.\ Math.\ Anal., { 34}, 64-86
(2002).

\bibitem{FuHaIn-16-DNLS}
Fukaya, N., Hayashi, M. and Inui, T.,
Global Well-Posedness on a generalized derivative nonlinear Schr\"odinger equation,  arXiv:1610.00267.

\bibitem{Fu-16-DNLS}
Fukaya, N., Instability of solitary waves for a generalized derivative nonlinear
Schr\"odinger equation in a borderline case, arXiv:1604.07945v2.

%



\bibitem{JeLiPeSu}
Jenkins, R., Liu, J., Perry, P.  and Sulem. C.,
Global existence for the derivative nonlinear Schr\"odinger equation with arbitrary spectral singularities. arXiv:1804.01506.



\bibitem{GuTa91}
Guo, B.  and Tan, S., {On smooth solution to the initial
value problem for the mixed nonlinear Schr\"odinger equations},
Proc. Roy. Soc. Edinburgh, 119, 31--45 (1991).

\bibitem{guo-DNLS}
Guo, Q.,
Orbital stability of solitary waves for generalized derivative nonlinear Schr\"odinger equations in the endpoint case, arXiv:1705.04458.

\bibitem{GuWu95}
Guo, B. and Wu, Y., Orbital stability of solitary waves for the nonlinear
derivative Schr\"odinger equation, J. Differential Equations, 123,
35--55 (1995).

\bibitem{GHLN}
Guo, Z., Hayashi, N.,  Lin, Y. and Naumkin, P., Modified scattering operator for the derivative nonlinear Schr\"odinger equation, Siam J. Math. Anal., 45 (6), 3854--3871 (2013).

\bibitem{Guo-Ning-Wu-18-gDNLS}
Guo, Z., Ning, C. and Wu, Y.,
{Instability of the solitary wave solutions for the genenalized derivative Nonlinear Schr\"odinger equation in the critical frequency case,} arXiv:1803.07700

\bibitem{Guo-Wu-15-DNLS}
Guo, Z. and Wu, Y.,
Global well-posedness for the derivative nonlinear Schr\"{o}dinger equation in $H^{\frac 12} (\R)$. Discrete Contin. Dyn. Syst. Ser. A, 37, 257-264, (2017).


\bibitem{Hao}
Hao, C., Well-posedness for one-dimensional derivative nonlinear Schr\"odinger equations, Comm.
Pure and Applied Analysis, 6(4), 997--1021 (2007).


\bibitem{HaOz-92-DNLS}
Hayashi, N. and Ozawa, T., {On the derivative nonlinear
Schr\"{o}dinger equation,} Physica\ D., {55}, 14--36 (1992).

\bibitem{HaOz-94-DNLS}
Hayashi, N. and Ozawa, T., {Finite energy solution of
nonlinear Schr\"{o}dinger equations of derivative type,} SIAM J.
Math. Anal., {25}, 1488--1503 (1994).

\bibitem{HaOz-16-DNLS}
Hayashi, M. and Ozawa, T., Well-posedness for a generalized derivative nonlinear
Schr\"odinger equation, J. Differential Equations 261, no. 10, 5424--5445 (2016).


\bibitem{Soonsik-W-2014}
Kwon, S. and Wu, Y., Orbital stability of solitary waves for derivative nonlinear Schr\"odinger equation,
preprint.

\bibitem{Stefan-W-15-MultiSoliton-DNLS}
Le Coz, S. and Wu, Y., Stability of multi-solitons for the derivative
nonlinear schr\"odinger equation, Inter. Math. Res. Notices, (2017).


\bibitem{LiPoSa-Local-DNLS}
Linares, L., Ponce, G. and Santos, G. N., On a class of solutions to the generalized
derivative schr\"odinger equations, preprint.


\bibitem{LiSiSu1}
Liu, X., Simpson, G. and Sulem, C., Stability of solitary waves for a generalized derivative
nonlinear Schr\"odinger equation, J. Nonlinear Science, 23(4), 557--583 (2013).


\bibitem{Miao-Tang-Xu18}
Miao, C., Tang, X. and Xu, G., Instability of the solitary waves for the generalized derivative nonlinear Schr\"odinger equation in the degenerate case, arXiv:1803.06451v1.


\bibitem{Miao-Tang-Xu:2016:DNLS}
Miao, C., Tang, X. and Xu, G., Stability of the traveling waves for the derivative Schr\"odinger
equation in the energy space. Calc. Var. Partial Differential Equations 56,  no. 2, Art. 45, 48 (2017).



\bibitem{Miao-Wu-Xu:2011:DNLS}
Miao, C., Wu, Y. and Xu, G., {Global well-posedness for
Schr\"{o}dinger equation with derivative in $H^{\frac 12} (\R)$,}
J.\ Diff.\ Eq., 251, 2164--2195 (2011).

\bibitem{MOMT-PHY}
Mio, W., Ogino,  T., Minami, K. and  Takeda, S., {Modified
nonlinear Schr\"{o}dinger for Alfv\'{e}n waves propagating along the
magnetic field in cold plasma,} J.\ Phys.\ Soc.\ Japan, {41},
265--271 (1976).

\bibitem{M-PHY}
Mjolhus, E., {On the modulational instability of hydromagnetic
waves parallel to the magnetic field,} J.\ Plasma\ Physc., {16},
321--334 (1976).


\bibitem{NOW16}
Ning, C., Ohta, M. and Wu, Y.,
{Instability of solitary wave solutions for derivative nonlinear Schr\"odinger equation in endpoint case,} J. Diff. Eq., 262, 1671--1689 (2017).

\bibitem{NOW17}
Ning, C., Ohta, M. and Wu, Y.,
{Instability of solitary wave solutions for derivative nonlinear Schr\"odinger equation in borderline case,} submit.


\bibitem{Ohta-14}
Ohta, M.,
{Instability of solitary waves for nonlinear Schr\"odinger equations of derivative type},
SUT J.\ Math., {50} (2), 399--415 (2014).

\bibitem{Oz-96-DNLS}
Ozawa, T., {On the nonlinear Schr\"{o}dinger equations of
derivative type,} Indiana Univ.\ Math.\ J., {45}, 137--163 (1996).

\bibitem{santos}
Santos, G. N., Existence and uniqueness of solutions for a generalized Nonlinear Derivative
Schr\"odinger equation, J. Diff. Eqs., 259, 2030--2060 (2015).




\bibitem{Ta-99-DNLS-LWP}
Takaoka, H., {Well-posedness for the one dimensional
Schr\"{o}dinger equation with the derivative nonlinearity,} Adv.\
Diff.\ Eq., {4 }, 561--680 (1999).

\bibitem{Ta-01-DNLS-GWP}
Takaoka, H., {Global well-posedness for Schr\"{o}dinger
equations with derivative in a nonlinear term and data in low-order
Sobolev spaces,} Electron.\ J.\ Diff.\ Eqns., {42}, 1--23 (2001).

%


\bibitem{TuFu-80-DNLS}
Tsutsumi, M. and Fukuda, I., On solutions of the derivative nonlinear Schr\"odinger equation. Existence and Uniqueness Theorem. Funkcial. Ekvac., 23, 259-277. (1980).

\bibitem{TuFu-81-DNLS}
Tsutsumi, M. and Fukuda, I., On solutions of the derivative nonlinear Schr\"odinger equation, II. Funkcial. Ekvac., 24, 85--94 (1981).



\bibitem{Wu1}
Wu, Yifei, Global well-posedness of the derivative nonlinear Schr\"odinger equations in energy space, Analysis \& PDE, 6 (8), 1989--2002 (2013).



\bibitem{Wu2}
Wu, Yifei,  {Global well-posedness on the derivative nonlinear Schr\"odinger
equation,} Analysis \& PDE, 8 (5), 1101--1113 (2015).


\end{thebibliography}
\end{document}